\documentclass[11pt]{article}
\usepackage{latexsym}
\usepackage{amsmath}
\usepackage{amssymb}
\usepackage{amsthm}
\def \Q {{\mathbb Q}}

\def \Z {{\mathbb Z}}

\def \M {{\mathbb M}}

\def \J  {{\mathbf J}}

\def \al {{\alpha}}
\def \be {{\beta}}
\def \ga {{\gamma}}
\def \Ga {{\Gamma}}

\def \ze {\zeta}

\newtheorem{theorem}{Theorem}
\newtheorem{cor}[theorem]{Corollary}
\newtheorem{lemma}[theorem]{Lemma}

\newtheorem{definition}[theorem]{Definition}
\newtheorem{rem}[theorem]{Remark}

\title{Non-singular circulant graphs and digraphs }

\author{ A. K. Lal\footnote{Department of Mathematics
$\&$ Statistics, Indian Institute of Technology Kanpur, Kanpur - 208
016, India. \newline E-Mails: arlal@iitk.ac.in; satya@iitk.ac.in.} \and A Satyanarayana Reddy$^*$ }

\date{}

\begin{document}
\maketitle

{\bf {Abstract }}
We give necessary and sufficient conditions for a few classes of known circulant graphs and/or digraphs to be singular. The above graph classes are generalized to $(r,s,t)$-digraphs for non-negative integers $r,s$ and $t$, and the digraph $C_n^{i,j,k,l}$, with certain restrictions. We also obtain a necessary and sufficient condition for the digraphs $C_n^{i,j,k,l}$ to be singular. Some necessary conditions are given under which the $(r,s,t)$-digraphs  are singular.


\section{Introduction and preliminaries}

Let $\Q$ denote the set of rational numbers.  Then the set
of all $n \times n$ matrices with entries from $\Q$ is denoted by
$\M_n(\Q)$. A matrix $A \in \M_n(\Q)$  is said to be symmetric if $A = A^t$, where $A^t$ denotes the transpose of the matrix $A$ and is said to be circulant if
$a_{ij}=a_{{1,\;j-i+1}}$,  whenever $2 \le i \le n$ and $1 \le j \le n$, where the subscripts are read modulo $n$.
From the definition, it is clear that if $A$ is circulant then for each $i\geq 2$ the elements of the $i$-th row are obtained by cyclically shifting the elements of
the $(i-1)$-th row one position to the  right. So it is sufficient
to specify its first row. For example, the identity matrix, denoted  $I$, and the matrix of all $1$'s, denoted $\J$, are circulant matrices. Let $W_n$ be a
circulant matrix of order $n$ with $[0, 1, 0, \ldots,
0]$ as its first row . Then the following result of Davis~\cite{davis} establishes that every circulant matrix of order $n$ is a polynomial in $W_n$.

\begin{lemma}\cite{davis}\label{lem:cir}
Let $A \in \M_n(\Q)$. Then $A$ is circulant if and only if it is a polynomial over $\Q$ in $W_n$.
\end{lemma}

Let $A \in M_n(\Q)$ be a circulant matrix. Then Lemma~\ref{lem:cir}
ensures the existence of a polynomial $\ga_A(x) \in \Q[x]$ such that
$A=\ga_A(W_n)$. We call $\ga_A(x)$, the representer polynomial of $A$. For a fixed positive integer $n$, let $\ze_n$ denote a primitive $n$-th root of unity. That is, $\ze_n^n = 1$ and $\ze_n^k \ne 1$ for $k=1, 2, \ldots, n-1$. Then the following result about circulant matrices is well known.

\begin{lemma}\label{lem:circu}
Let $A \in M_n(\Q)$ be a circulant matrix with $[a_0, a_1,\ldots,
a_{n-1}]$ as its first row. Then
\begin{enumerate}
\item  $\ga_A(x)= a_0+a_1x+\cdots +a_{n-1}x^{n-1} \in \Q[x]$.
\item the eigenvalues of $A$ are given by $\ga_A(\ze_n^k)$, for $ k = 0, 1,
\ldots, n-1$.
\end{enumerate}
\end{lemma}

For definitions and results related to linear algebra, algebra and/or graph theory  that have been used in this paper but not have been cleared defined or stated, the readers are advised to see any standard textbook on abstract algebra and/or graph theory(for example, see~\cite{D:F} and/or~\cite{bondy}).

Recall that a  {\em directed graph} (in short, digraph) is an ordered pair $X=(V,E)$ that consists of two sets
$V$, the vertex set, and $E$, the edge set, where $V$ is a non-empty set and  $E \subset V \times V$. If $e=(u, v) \in E$ with $u \ne v$ then the edge $e$ is
said to be {\em incident} from  $u$ to $v$. A digraph is called a  {\em graph} if $(u,v) \in E$ whenever $(v,u) \in E$, for any two elements $u,v \in V$.  An edge between $u$ and $v$ in the graph $X$ is denoted by $\{u,v\}$. A graph/digraph is said to be finite, if  $|V| $ (called the {\em order} of $X$) and $|E|$ (called the {\em size} of $X$) are finite. All the graphs/digraphs in this paper are finite. The adjacency matrix of a graph/digraph $X = (V,E)$ is a $|V| \times |V|$ matrix, denoted $A(X) = [a_{uv}]$, with $a_{uv} = 1$ if $(u,v) \in E$ and $0$, otherwise. Observe that, whenever $X$ is a graph the matrix $A(X)$ is symmetric.
For example, if $A$ denotes the adjacency matrix of the cycle graph $C_n$ on $n$
vertices, then $A$ is a circulant matrix and $\ga_{A}(x)=x+x^{n-1}$ is its
representer polynomial. Therefore, for $r=0,1,\ldots,
n-1$, the eigenvalues of $C_n$ are given by $\lambda_r=2 \cos(\frac{2\pi r}{n})$.
Throughout this paper, we assume that the greatest common
divisor, in short $\gcd$, of all the non-zero coefficients of $\ga_A(x)$ is $1$.
It is well known that $x^n-1=\prod\limits_{d \mid n}\Phi_d(x)$ (here $a \mid b$ means $a$ `divides' $b$), where $\Phi_d(x)=\prod\limits_{\substack{\gcd(k,d)=1\\ 1 \le k \le d}}(x-\zeta_d^k)\in\Z[x]$  is called the $d$-th cyclotomic polynomial. The polynomial $\Phi_n(x)$, for each positive integer $n$, is a monic  irreducible polynomial over $\Q$ and hence  the minimal polynomial of $\ze_n$. Also,   $\deg(\Phi_n(x)) = \varphi(n)$, the well known Euler-totient function. Therefore, using the property of minimal polynomials, it follows
that  if $f(\ze_n) = 0$ for some $f(x) \in \Z[x]$ then $\Phi_n(x)$ divides $f(x) \in \Z[x]$. Or equivalently, $f(\ze_n) = 0$ for some $f(x) \in \Z[x]$ if and only if there exists a polynomial $g(x) \in \Z[x]$ such that $f(x) = \Phi_n(x)
g(x)$.  The next result appears on page $93$ in~\cite{V:P}.

\begin{lemma}\cite{V:P} \label{lem:cyclo}
Let $p$ be a prime number and let $n$ be a positive integer. Then
$$\Phi_{pn}(x)=
\begin{cases} \Phi_n(x^p), & \mbox{ if } p\mid n,\\
\frac{\Phi_n(x^p)}{\Phi_n(x)}, & \mbox{ if } p\nmid n.
\end{cases}$$
In particular, $\Phi_{p^k}(x)=1+x^{p^{k-1}}+x^{2p^{k-1}}+\dots
+x^{(p-1)p^{k-1}}$ for every positive integer $k$.
\end{lemma}

 The following result is an application of Lemma~\ref{lem:circu}. This result also appears in the work of Geller, Kra, Popescu $\&$ Simanca~\cite{gel}.

\begin{lemma}[Geller, Kra, Popesu $\&$ Simanca~\cite{gel}]\label{lem:singular}
 Let $A \in \M_n(\Q)$ be a circulant matrix with $\ga_A(x)$ as
its representer polynomial. Then the following statements are
equivalent:
\begin{enumerate}
\item The matrix $A$ is singular.
\item $\deg( \gcd( \ga_A(x), x^n-1 ) ) \ge 1$.
\end{enumerate}
\end{lemma}

Fix a positive integer $n$, two distinct integers $a$ and $b$ and let $s$ and
$t$ be positive integers with $s+t = n$. Suppose $[\; \underbrace{a, a, \ldots,
a}_{s {\mbox{ times}}},\;\underbrace{b, b, \ldots,  b}_{t {\mbox{ times}} } \;]$
is the first row of a circulant matrix $A \in \M_n(\Z)$. Then as a direct corollary of  Lemma~\ref{lem:singular}, one has the following result.
\begin{cor}[Davis~\cite{davis}]\label{cor:ab}
Let $[\; \underbrace{a, a, \ldots,
a}_{s {\mbox{ times}}},\;\underbrace{b, b, \ldots,  b}_{t {\mbox{ times}} } \;]$ be the first row of the circulant matrix $A \in \M_n(\Q)$.  Then
$$\det(A)=\begin{cases} (sa+tb)(a-b)^{n-1}, &
\mbox{if} \;\gcd(s,n)=1,\\
0, & \;\mbox{otherwise}.
\end{cases}$$
\end{cor}

We now   state a couple of known results that directly follow from Corollary~\ref{cor:ab}.
\begin{lemma}\label{lem:K_n}
The complete graph $K_n$, for $n \ge 2$, is non-singular.
\end{lemma}

\begin{proof}
Let  $A$ be the adjacency matrix of complete graph $K_n$.
 Then $[0,1,1,\ldots,1]$ is the first row of $A$. Hence the result follows from
Corollary~\ref{cor:ab}.
 \end{proof}

As a second application, we consider a
particular  class of circulant matrices that appeared in the work of
Searle~\cite{S:R}. He considered the circulant matrices that have  $[ h_0, h_1, \ldots, h_{k-1}, \underbrace{0, \ldots,
0}_{n-k}]$ as its first row, where  $h_0 \ne 0 $ and $h_{k-1} \ne 0$. The above class of matrices was called a {\em $k$-element circulant
matrix}\index{$k$-element circulant
matrix}. Since we are looking at digraphs, we assume $h_0 = 1 =
h_{k-1}$. With an abuse of notation, the circulant matrix with
$[\underbrace{1, 1, \ldots, 1}_k,\;\underbrace{0, 0, \ldots,
0}_{n-k}]$ as its first row will be called a {\em $k$-element
circulant digraph}.   With this notation, the second application of
Corollary~\ref{cor:ab} is stated below.

\begin{lemma}\label{lem:k}
Let $X$ be a $k$-element circulant digraph on $n$ vertices. Then $X$ is
non-singular if and only if  $\gcd(n,k)=1$.
\end{lemma}

We now rephrase Lemma~\ref{lem:singular} in terms of cyclotomic polynomials.
Let $A \in \M_n(\Z)$ be a circulant matrix with $[a_0, a_1, \ldots, a_{n-1}]$ as its first row. Then $\ga_A(x)=a_0+a_1x+\cdots +a_{n-1}x^{n-1}$ is the representer polynomial of $A$.  Now, suppose that $a_0 = 0$ and let  $k$ be the smallest positive integer
such that $a_k\neq 0$. Then   $\ga_A(x)=x^k \Ga_A(x)$, for some
polynomial $\Ga_A(x) \in \Z[x]$. In this case, it follows that the matrix
$A$ is non-singular if and only if $\gcd(\Ga_A(x), x^n-1)=1$ as
$\gcd(x^k, x^n-1)=1$. This observation leads to the next remark.

\begin{rem}\label{rem:Wn}
Let $A \in M_n(\Z)$ be  a circulant matrix and for each fixed
positive integer $k$ consider the matrix $W_n^k A$. Then $A$ is
singular (non-singular) if and only if $W_n^k A$ is
singular (non-singular). That is, if we want to study
singularity/non-singularity of a matrix $A$ then it is enough to study $\Ga_A(x)$.
\end{rem}

Using Remark~\ref{rem:Wn} and Lemma~\ref{lem:singular}, the following result is immediate and hence the proof is omitted.

\begin{lemma}\label{lem:Ga}
Let $A$ be a circulant digraph of order $n$ and let $\Ga_A(x)$ be the polynomial
defined above. Then $A$ is singular if and only if
$\Phi_d(x) \mid \Ga_A(x)$, for some divisor $d\neq 1$ of $n$.
\end{lemma}

As an immediate corollary of Lemma~\ref{lem:Ga}, we have the following result.

\begin{cor}\label{cor:prime}
Let $p$ be a prime and let $k$ be a positive integer with $p\nmid k$. Also, let
$X$ be a $k$-regular circulant graph/digraph on
$p^{\ell}$ vertices, for some positive integer $\ell$. Then $X$ is
non-singular.
\end{cor}

\begin{proof}
Using Lemma~\ref{lem:Ga}, we just need to show that $\Phi_d(x)\nmid
\Ga_A(x)$ for every  $d \mid p^{\ell}, d \ne 1$.
Let if possible, $\Ga_A(x) = \Phi_d(x) g(x)$ for some $g(x) \in \Z[x]$. Using
Lemma~\ref{lem:cyclo}, we have $\Phi_d(1)=p$ for every  $d \mid p^{\ell}, d
\ne 1$. As $g(x) \in \Z[x], g(1) \in \Z$. Thus,  we get
$$ k = \Ga_A(1) = \Phi_d(1) g(1) = p \; g(1).$$ A
contradiction to our assumption that $p \nmid k$. Thus the proof of
the result is complete.
\end{proof}

The remaining part of this paper consists of two more sections that are mainly concerned with applications of Lemma~\ref{lem:Ga}. Section~\ref{sec3:2}
gives necessary and sufficient conditions for a few classes of circulant
graphs  to be non-singular and
Section~\ref{sec3:3} gives possible generalization of the results
studied in Section~\ref{sec3:2}.

Before proceeding to Section~\ref{sec3:2}, recall that for a graph
$X=(V,E_1)$,  the complement graph of $X$, denoted $X^c= (V,E_2)$,  is a graph in which $(u,v) \in E_2$ whenever $(u,v) \not\in E_1$ and vice versa, for every $u \ne v \in V$. Note that $(u,u)$ is neither an element of $E_1$ nor an element of $E_2$. Also,  a graph  $X$ is circulant if and only if $X^c$ is  circulant and if $A$ is the adjacency matrix of $X$ then the adjacency matrix of $X^c$ is given by $\J - A
- I$.

\section{Some Singular Circulant Graphs}\label{sec3:2}
This section is devoted to finding necessary and sufficient conditions for a few
classes of circulant graphs  to be singular or not. Before proceeding to these results, we show that the adjacency matrix $A$ of a circulant graph on $n$ vertices is a polynomial in $W_n + W_n^{-1}$, the adjacency matrix of $C_n$, the cycle graph on $n$ vertices. To do so, we need the following definition.

\begin{definition}\label{def:dis}
 Let $v_1, v_2, \ldots, v_n$ be the vertices of a connected graph $X$. If $d$ is the diameter of $X$ then, for  $0\le k \le d$, the {\em  $k$-th distance matrix}\index{$k$-th distance matrix} of $X$, denoted
$A_k(X)$,  is defined as
$$ (A_{k}(X))_{rs} = \left\{ \begin{array}{cl} 1, & {\mbox{ if }} d(v_{r},v_{s}) =
k,  \\ 0, & {\mbox{ otherwise,}} \end{array} \right. $$
where $d(u,v)$ is the distance between the vertices $u, v \in V$.
\end{definition}

For example, $\tau = \lfloor \frac{n}{2} \rfloor$ and consider the cycle graph $C_n$. Also, let us write $A_k$ to denote the distance matrices $A_k(C_n)$, for $0 \le k \le \tau$. Then, for  $1\le i <  \tau, \; A_i = W_n^i + W_n^{n-i}$ and
\begin{equation}\label{eqn:dis:cn}
 A_{\tau} = \begin{cases} W_n^{\tau}, &
\mbox{ if } n \mbox{ is even},\\
W_n^{\tau} + W_n^{n-\tau}, & \mbox{ if } n \mbox{ is odd}.
\end{cases}
\end{equation}

The identity $$(x^k+x^{-k})=(x + x^{-1}) (x^{k-1}+x^{1-k})-(x^{k-2}+x^{2-k})$$
enables us to  readily establish, by mathematical induction, that $x^k+x^{-k}$ is a monic polynomial in $x+x^{-1}$ of degree $k$ with integral coefficients. Also, for $n$ even, $2 \tau = n$ and hence $W_n^{\tau} = \frac{W_n^{\tau} + W_n^{-\tau}}{2}$. Consequently, $A_i$'s, for $1\le i \le \tau$, are polynomials of degree
$\leq i$, in $A$, the adjacency matrix of $C_n$, over $\Q$. Now, let $B$ be a  symmetric circulant matrix with  representer
polynomial $\ga_B(x)=\sum\limits_{i=0}^{n-1}b_ix^i$. Then by definition,
$B=\ga_B(W_n) = \sum\limits_{i=0}^{n-1}b_iW_n^i$ and $B^t=\sum\limits_{i=0}^{n-1}b_iW_n^{n-i}$. But $B$ is symmetric implies that $B = B^t$ and therefore,  $b_i=b_{n-i}$, for $1\leq i\leq n-1$. Thus,
$B=\sum\limits_{i=0}^{\tau}b_iA_i$ and hence we see that the adjacency matrix of any circulant graph is a polynomial in $A$, the adjacency matrix of $C_n$, over $\Q$.

For $ 1 \le i \le \tau$, let us denote the graph with  adjacency matrix $A_i$  as $C_n^i$. Then observe  that $C_n^1 = C_n$ is the cycle
graph on $n$ vertices. Also, note that the corresponding representer
polynomials, for $1 \le i < \tau$, are given
by $\ga_{A_i}(x) = x^i \left( 1 + x^{n-2 i} \right)$ and
$ \ga_{A_{\tau}}(x) = x^{\tau}$, if $n$ is even, and $\ga_{A_{\tau}}(x) = x^{\tau}(1 + x)$, if $n$ is
odd. The next result uses the above notations and observations to
give a necessary and sufficient condition for the graphs $C_n^i$, for
$1 \le i \le \tau$, to be singular.

\begin{lemma}\label{lem:cni}
Fix a positive integer $n \ge 3$ and let $1 \le i \le \tau=\lfloor \frac{n}{2}
\rfloor.$ Then the  graph $C_n^i$  is singular if and only if  $n$ is a multiple
of $4$ and $\gcd(i, \frac{n}{2}) \mid \frac{n}{4}$.
\end{lemma}

\begin{proof}
Using the discussion above, $\Ga_{A_i}(x)= 1+ x^{n-2i}$, for $1 \le i < \tau$ and
$\Ga_{A_{\tau}}(x) = 1 + x$, if $n$ is odd, and $\Ga_{A_{\tau}}(x) = 1$, if $n$ is
even. If $n$ is odd then $\ze_n^k \ne -1$ for any $k, 1 \le k \le n-1$. Hence, $A_{\tau}$ is non-singular for all $n$. So, we
need to consider $\Ga_{A_i}(x)= 1+ x^{n-2i}$, for $1 \le i < \tau$.

In this case, $A_i$ is singular if and only if $\Ga_{A_i}(\ze_n^k) = 0$,
for some $k, 1 \le k \le n-1$. That is, we need
$\left(\ze_n^k\right)^{n-2 i} = -1$. Or equivalently, we need
$2 k i \equiv \frac{n}{2} \pmod n $. That is, $k i \equiv
\frac{n}{4} \pmod {\frac{n}{2}}.$ Therefore, it follows that $4$ divides $n$ and
$\gcd(i, \frac{n}{2}) \mid \frac{n}{4}$.
\end{proof}

\begin{rem}
We can rewrite the condition in Lemma~\ref{lem:cni} as follows: \\
The  graph $C_n^i$  is
singular if and only if  the following conditions are satisfied:
\begin{enumerate}
\item $n$ is a multiple of $4$, and
\item if $s$ is the largest positive integer such that $2^s$ divides $n$ then
$i$ is an odd multiple of $2^{t}$ for some  $t, \; 0 \le t \le s-2$.
\end{enumerate}
\end{rem}

As an immediate consequence of Lemma~\ref{lem:cni}, we have the following
corollary.
\begin{cor}
Let $C_n$ be the cycle graph on $n$ vertices. Then  $C_n$ is singular if and
only if $4 \mid n$.
\end{cor}

The next result gives a necessary and sufficient condition for the complement
graph $(C_n^i)^c$ of $C_n^i$ to be singular.

\begin{lemma}\label{lem:cnic}
Fix a positive integer $n \ge 4$. Then the graph
\begin{enumerate}
\item $(C_n^{\tau})^c$ is singular if and only if $n$ is even or
$n \equiv 3 \pmod 6$.
\item $(C_n^i)^c$, for $1 \le i < \tau$,  is singular if
and only if  $3 \mid n$ and $\gcd(i,n) \mid \frac{n}{3}.$
\end{enumerate}
\end{lemma}

\begin{proof}
Let $n$ be even. Then, using definition of complement of a graph, the adjacency matrix of $(C_n^{\tau})^c$, say $A$, is given by $\J  - A_{\tau} - I$. Hence, $\ga_A(x) = 1 + x + \cdots + x^{n-1} - 1 - \ga_{A_{\tau}}(x)$ is the representer polynomial of $(C_n^{\tau})^c$.  Thus,
$$\ga_A(x) = \frac{x^n - 1}{x-1} - (1 + x^{\tau})$$ and hence $\ga_A(\ze_n) =
0$ ($n$ is even). Thus, the graph $(C_n^{\tau})^c$
is singular, whenever $n$ is even. For $n$ odd, it can be checked that
$$\ga_A(x) = \frac{x^n - 1}{x-1} - (1  + x^{\tau} + x^{\tau+1}).$$
Consequently, $(C_n^{\tau})^c$ is singular if and only if $\ga_A(\ze_n^k) =0$,
for some $k, 1 \le k \le n-1$. Or equivalently, $1 + (\ze_n^k)^{\tau} +
(\ze_n^k)^{\tau+1} = 0$, for some $k, 1 \le k \le n-1$. This is equivalent to the
statement that $\ze_n^{k \tau}$ is a primitive $3$-rd root of unity. Thus, $k
\tau \equiv \frac{n}{3} \pmod n$, or equivalently, $\gcd( \tau, n) \mid
\frac{n}{3}$. Thus, $n \equiv 3 \pmod 6$ and in this case, $\gcd( \tau, n)$ indeed divides $\frac{n}{3}$.

Now assume that $1 \le i < \tau.$ In this case, if $A$ is the
adjacency matrix of $(C_n^i)^c$, then $A=\J - A_i -
I$. Consequently, its representer polynomial is
$\ga_A(x)=\frac{x^n-1}{x-1}-(1+x^i+x^{n-i})$.  Thus,
$(C_n^i)^c$ is singular if and only if $\ga_A(\ze_n^k)=0 $, for some
$k, 1 \le k \le n-1$. Or equivalently, $1+\ze_n^{ki}+\ze_n^{- ki}=0$, for some $k$, $1 \le k \le n-1$. That is, $\ze_n^{k i}$ is a primitive $3$-rd root of unity.
Thus,  using the argument similar to one in the first part, one has
$(C_n^i)^c$ is singular if and only if
 $3 \mid n$ and $\gcd(i,n) \mid \frac{n}{3}.$
\end{proof}

As an immediate consequence of Lemma~\ref{lem:cnic}, we have the following
corollary.

\begin{cor}
Fix a positive integer $n$ and let  $C_n^c$ be the complement graph of the cycle
graph $C_n$. Then the complement graph $C_n^c$ is singular if and only if $3
\mid n$.
\end{cor}

We now obtain necessary and sufficient conditions for
non-singularity of circulant graphs that were studied by Ruivivar~\cite{leo}. In~\cite{leo}, the author studied two classes of graphs. For the sake of notational clarity, his notations have been
slightly modified. Fix a positive integer $n \ge 3$ and let  $1 \le
r < \tau=\lfloor{\frac{n}{2}}\rfloor$. The first class of circulant
graphs, denoted $C_n^{(r)}$, has the same vertex set as the vertex
set of the cycle $C_n$ and  $\{x, y\}$ is an edge whenever the
length of the smallest path from $x$ to $y$ in $C_n$ is at most $r$.
He called these graphs the $r$-th power graph of the cycle graph
$C_n$. Note that $C_n^{(\tau)}$ is the complete graph.
The second class of graphs, denoted
$C(2 n,r)$ is a graph on $2 n$ vertices and its adjacency matrix is
the sum of the adjacency matrices of  $C_{2n}^{(r)}$ and $C_{2n}^n$,
where  $1\leq r < n$. The next result appears as Theorem~2.2
of~\cite{leo}. We give a separate proof  for the sake of
completeness.

\begin{theorem}[Ruivivar~\cite{leo}]
Let $n\geq 3$ and let $1\leq r < \lfloor{\frac{n}{2}}\rfloor$. Then
the graph $C_n^{(r)}$ is singular if and only if one of the
following conditions hold: \begin{enumerate}
\item  $\gcd(n,r)>1$
\item $\gcd(n,r)=1$,  $n$ is even and $\gcd(r+1,n)$ divides $\frac{n}{2}$.
\end{enumerate}
\end{theorem}

\begin{proof}
Let $A$ be the adjacency matrix of the graph $C_n^{(r)}$. Then, by definition,
the first row of $A$ equals $[0,\underbrace{1,1,\ldots,
1}_r\;\underbrace{0,0,\ldots, 0}_{n-2r-1}\;\underbrace{1,1,\ldots,
1}_r]$ and $\ga_A(x)= x\Ga_A(x)$, where
\begin{equation*}\Ga_A(x)=[1+x+\dots +x^{r-1}]+x^{n-r-1}[1+x+\dots +x^{r-1}] =
\frac{x^r-1}{x-1} (1+x^{n-r-1}).
\end{equation*}
Therefore,  $C_n^{(r)}$ is singular if and only if
$\Ga_A(\ze_n^d)=0$, for some $d, 1\leq d\leq n-1$. Or equivalently either
$(\ze_n^d)^r-1=0$ or $1+(\ze_n^d)^{n-r-1}=0$.

If $(\ze_n^d)^r-1=0$  then  $\gcd(r,n)>1$ is the required condition as
$1\leq d\leq n-1$. If $\gcd(r,n)=1$ then we need $1+(\ze_n^d)^{n-r-1}=0$. This
implies that $d(r+1) \equiv \frac{n}{2} \pmod n$. Which in turn gives the
required result.

Thus, the proof of the theorem is complete.
\end{proof}

The following result can be seen as a corollary to
Lemma~\ref{lem:k}. But an idea of the proof is given for completeness.

\begin{cor}
Let $n\geq 3$ and let $1\leq r < \lfloor{\frac{n}{2}}\rfloor$. Then the
graph $(C_n^{(r)})^c$ is non-singular if and only if
$\gcd(n,2r+1)=1$.
\end{cor}

\begin{proof}
Let $A$ be the adjacency matrix of $(C_n^{(r)})^c$. Then $[\underbrace{0,0,\ldots,
0}_{r+1}\;\underbrace{\;1,\ldots, 1}_{n-2r-1}\;\underbrace{0,0,\ldots,
0}_r]$ is the first row of $A$. Thus, using Remark~\ref{rem:Wn}, $A$ is singular
if and only
if  the circulant matrix with
$[\underbrace{1,1,\ldots,1}_{n-2r-1}\;\underbrace{0,0,\ldots,
0}_{2r+1}]$ as its first row is singular. Thus, using Lemma~\ref{lem:k}, $A$ is
singular if and only if
 $\gcd(2r+1,n)>1$. Hence,  the required result follows.
\end{proof}

Before proceeding with the next result that gives a necessary and
sufficient condition for the graph $C(2n, r)$ to be singular, we
state a result that appears as Proposition~$1$ in  Kurshan $\&$  Odlyzko~\cite{odlyzko}

\begin{lemma}[Kurshan $\&$ Odlyzko~\cite{odlyzko}]\label{lem:odlyzko}
Let $m$ and $n$ be positive integers with $m \ne n$ and let $\ze_n$
be a primitive $n$-root of unity. Then there exists a unit $u \in
\Z[\ze_n]$ dependent on $m, n$ and $\ze_n$ such that
$$\Phi_m(\ze_n) = \left\{ \begin{array}{lllll}
p u, & {\mbox{ if }} \frac{m}{n} = p^\al, & p {\mbox{ a prime,}} &
\al > 0; &  \\ (1 - \ze_{p^\al}) u, & {\mbox{ if }} \frac{m}{n} =
p^{-\al}, & p {\mbox{ a prime,}} & \al > 0; &  p \nmid m; \\
(1 - \ze_{p^\al+1})^{p-1} u, & {\mbox{ if }} \frac{m}{n} =
p^{-\al}, & p {\mbox{ a prime,}} & \al > 0; &  p \mid m; \\
u, & {\mbox{ otherwise.}} & & & \end{array}\right.$$
\end{lemma}

\begin{theorem}\label{thm:cnr}
Let $n$ and $r$ be positive integers such that the circulant graph
$C(2n,r)$ is well defined. Then the circulant graph $C(2n,r)$ is
singular if and only if  $\gcd(n,2r+1)\ge 3$.
\end{theorem}

\begin{proof}
Let $A$ be the adjacency matrix of the graph $C(2n,r)$. Then observe that the
first row of $A$ equals
$[0, \underbrace{1, 1, \ldots, 1}_r, \underbrace{0, 0, \ldots,
0}_{n-r-1}, 1, \underbrace{0, 0, \ldots, 0}_{n-r-1},
\underbrace{1, 1, \ldots, 1}_r]$.
 Consequently, $$\ga_A(x)=x+x^2+\dots +x^r+x^n+x^{2n-r}+\dots +x^{2n-1}= x
\Ga_A(x)$$ and
\begin{eqnarray*}
(x-1) \Ga_A(x)&=& x^r- 1 + x^{n-1}(x-1) + x^{2n-r-1}(x^r - 1)\\
&=& x^r (1 - x^{2n - 2r - 1}) + (x^{n} -1 ) - (x^{n-1} - x^{2n-1})
\\  &=& (x^n-1)(x^{n-1}+1)-x^r(x^{2n-2r-1}-1).
\end{eqnarray*}
Now, let us assume that $\gcd(n,2r+1) = d \ge 3$. Then  $(\ze_{2n}^{ 2 n/d} -
1) \Ga_A(\ze_{2n}^{2
n/d}) = 0$ as $$\left(\ze_{2n}^{2 n/d}\right)^{n}=
\left(\ze_{2n}^{2 n}\right)^{n/d} = 1 =  \left(\ze_{2n}^{2
n}\right)^{(2 r + 1)/d} = \left(\ze_{2n}^{2 n/d}\right)^{2 r + 1}=
\left(\ze_{2n}^{2 n/d}\right)^{2n - 2 r - 1}.$$ Hence, the circulant graph $C(2n,r)$ is singular.

Conversely, let us assume that the graph $C(2n,r)$ is singular. This
implies that there exists an eigenvalue of $C(2n,r)$ that equals
zero. That is, there exists a $k, \; 1 \le k \le 2n - 1$, such that
$\ga_A(\ze_{2n}^k) = 0$. We will now show that if
$\gcd(n, 2r + 1) = 1$ then the expression $(x-1)\Ga_A(x)$ evaluated
at $x = \ze_{2n}^k$ can never equal zero, for any $k, \; 1 \le k
\le 2n-1$, and this will complete the proof of the result.

We need to consider two cases depending on whether $k$ is odd or $k$
is even. Let $k$ be even, say $k = 2m$, for some  $m, \; 1 \le m < n$. Then
evaluating $(x-1) \Ga_A(x)$ at $x = \ze_{2n}^{2m}$ and using
$\gcd(n, 2r + 1) = 1$ leads to
\begin{eqnarray*}
 && \hspace{-1in} \left[ \left(\ze_{2n}^{2m} \right)^n - 1\right]
\left[ \left(\ze_{2n}^{2m} \right)^{(n-1)} + 1\right] -
\left(\ze_{2n}^{2m} \right)^r \left[ \left(\ze_{2n}^{2m}
\right)^{2n- 2r - 1} - 1\right] \\ &=& - \left(\ze_{2n}^{2m}
\right)^r \left[ \left(\ze_{2n}^{2m} \right)^{-( 2r + 1)} -
1\right] \ne 0.
\end{eqnarray*}
Now, let $k$ be odd, say $k = 2m + 1$, for some $m, \; 0 \le m \le n-1$. Then
evaluating $(x-1) \Ga_A(x)$ at $x = \ze_{2n}^{2m+1}$ leads to
\begin{eqnarray}
 && \hspace{-1in} \left[ \left(\ze_{2n}^{2m+1} \right)^n - 1\right]
\left[ \left(\ze_{2n}^{2m+1} \right)^{(n-1)} + 1\right] -
\left(\ze_{2n}^{2m+1} \right)^r \left[ \left(\ze_{2n}^{2m+1}
\right)^{2n- 2r - 1} - 1\right] \nonumber \\
&=& - 2 \left[ - \ze_{2n}^{-(2m+1)} + 1  \right] -
\ze_{2n}^{-(2m+1) (r+1)} \left[ 1 -
\ze_{2n}^{(2m+1) (2r + 1)} \right] \nonumber \\
&=& -\frac{ \ze_{2n}^{2m+1} - 1}{\ze_{2n}^{(2m+1)(r+1)}} \left[ -
2 \ze_{2n}^{(2m+1) r } + \frac{ \ze_{2n}^{(2m+1)(2r + 1)} - 1}{
\ze_{2n}^{(2m+1)} - 1} \right] \nonumber \\
&=& -\frac{ \ze_{2n}^{2m+1} - 1}{\ze_{2n}^{(2m+1)(r+1)}} \left[ -
2 \ze_{2n}^{(2m+1) r } + \prod_{\ell \mid (2r+1), \ell \ne 1}
\Phi_{\ell}(\ze_{2n}^{2m+1}) \right] \label{eqn:odd}
\end{eqnarray}
Note that, $\ze_{2n}^{2m+1}$ is a $d$-th primitive root of unity, for some $d$
dividing $2n$. As $\gcd(2 r + 1, 2n) = 1$, $\gcd(2 r + 1, d) = 1$.
Thus,  using Lemma~\ref{lem:odlyzko}, we get
$\prod\limits_{\ell \mid (2r+1), \ell \ne 1}
\Phi_{\ell}(\ze_{2n}^{2m+1})$ is a unit in $\Z[\ze_{d}]$. That
is, $\left|\prod\limits_{\ell \mid (2r+1), \ell \ne 1}
\Phi_{\ell}(\ze_{2n}^{2m+1})\right| = 1$.
 Hence, in Equation~(\ref{eqn:odd}), the
term in the parenthesis cannot be zero. Thus, we have proved the
result for the odd case as well.

Thus, the proof of the result is complete.
\end{proof}

\begin{rem}
We would like to mention here that the necessary part of Theorem~\ref{thm:cnr}
was stated and proved by Ruivivar~(see Theorem~2.1 in~\cite{leo}).
\end{rem}

We will now try to understand the complement graph $C(2n,r)^c$ of $C(2n,r)$.
\begin{lemma}\label{lem:cnrc}
Let $n$ and $r$ be positive integers such that the circulant graph
$C(2n,r)$ is well defined. Then $C(2n,r)^c$ is non-singular if and only if the
following conditions hold:
\begin{enumerate}
\item $n$ and $r$ have the same parity,
\item $\gcd(n, r+1) = 1$, and
\item the highest power of $2$ dividing $n$ is strictly smaller than the highest power of $2$ dividing $n-r$.
    \end{enumerate}
\end{lemma}

\begin{proof}
 Let $A$ be the adjacency matrix of $C(2n,r)^c$. Then $[\underbrace{0,0,\ldots,
0}_{r+1}\;\underbrace{1,1,\ldots, 1}_{n-r-1}\;0\;\underbrace{1,1,\ldots,
1}_{n-r-1}\;\underbrace{0,0,\ldots, 0}_r]$ is the first
 row of $A$. Note that $$\Ga_A(x) =  (1 + x^{n-r}) \frac{x^{n-r-1} - 1}{x - 1}.$$

Now, let us assume that the graph $C(2n, r)^c$ is non-singular. This means that
$\Ga_A(\ze_{2n}^{k}) \ne 0$, for any $k = 1, 2, \ldots, 2n-1$.

Note that if $n$ and $r$ have opposite parity then $\gcd(2n, n-r - 1) = d \ge 2$
and hence $\Ga_A(\ze_{2n}^{2n/d}) = 0$. Also, if $n$ and $r$ have the same
parity and $\gcd(n,r+1) = d > 2$ then  $n - r - 1$ is odd  and  $\gcd(2n, n-r -
1) = \gcd(n,n-r-1) = \gcd(n,r+1) = d$. Hence, in this case again,
$\Ga_A(\ze_{2n}^{2n/d}) = 0$.

Now, the only case that we need to check is the following: \\
$n$ and $r$ have the same parity, $\gcd(n, r+1) = 1$ and the highest power of
$2$ dividing $n$ is greater than or equal to the highest power of $2$ dividing
$n-r$.

As $n$ and $r$ have the same parity and $\gcd(n,r+1) = 1$, we get $\gcd(2n,
n-r-1) = 1$ and thus $$(\ze_{2n}^k)^{n-r-1} - 1 \ne 0, \; {\mbox{ for any }} \;
k = 1, 2, \ldots, 2n-1.$$
Thus, we need to check for the condition on $k$ so that $1 + (\ze_{2n}^k)^{n-r}
= 0.$ This is true if and only if $\gcd(2n, n-r) \mid n$, or equivalently, the
highest power of $2$ dividing $n$ is greater than or equal to the highest power
of $2$ dividing $n-r$.

Thus, we have the required result.
\end{proof}

\begin{rem}
Observe that using Lemma~\ref{lem:cnrc}, the graph $C(2n, r)^c$ is non-singular,
whenever $n$ and $r$ are both odd and $\gcd(n, r+1) = 1$. Such numbers can be
easily computed. For example, a class of such graphs can be obtained by choosing
two positive integers $s$ and $t$ with $s > t$ and defining $n = 2^s - 2^t + 1$
and $r = 2^t - 1$.
\end{rem}

\section{Generalizations}\label{sec3:3}
In this section, we look at a few classes of graphs/digraphs, which are
generalizations of the graphs that appear in Section~\ref{sec3:2}. We first
start with a class of circulant digraphs.

Consider a circulant matrix $A$ whose first row
contains $r$  and $s$ consecutive $1$'s separated by $t$ consecutive
$0$'s, where each of $r, s$ and $t$ are non-negative integers. That is, the vector $[\underbrace{1, 1,
\ldots, 1}_r, \; \underbrace{0, 0, \ldots, 0}_t, \;\underbrace{1, 1,
\ldots, 1}_s,\; \underbrace{0, 0, \ldots, 0}_{n-(r+t+s)}]$ is the first row of  $A$. If $s=0$, then it is an $r$-element circulant digraph studied in Lemma~\ref{lem:k}.
These circulant digraphs will be called an $(r,s,t)$-element circulant digraph.
The next result
gives a few conditions under which the $(r,s,t)$-element circulant
digraph is singular.

\begin{lemma}\label{lem:rst}
Let $X$ be an $(r,s,t)$-element circulant digraph on $n$ vertices.
 Then the graph $X$ is singular if
\begin{enumerate}
\item \label{lem:rst:1} $\gcd(n,s,r) > 1$, or
\item \label{lem:rst:2} $\gcd(n,s)= 1$ and one of the following condition holds:
\begin{enumerate}
\item \label{lem:rst:2.1} there exists  $d \ge 2$ such that $d \mid t$ and
$s=\ell r$, for some positive integer $\ell \equiv -1 \pmod d$.
\item \label{lem:rst:2.2} $n$ is even, there exists an even integer $d$ such that $(r + t)$ is an odd multiple of $\frac{d}{2}$ and $s=\ell r$, for some positive integer $\ell \equiv 1 \pmod d$.
\end{enumerate}
\end{enumerate}
\end{lemma}

\begin{proof}
{\bf Proof of Part~\ref{lem:rst:1}:} Observe that the  representer polynomial of the $(r,s,t)$-element  circulant digraph is given by
\begin{eqnarray*}
\ga_A(x)&=& 1+x+\dots +x^{r-1}+x^{r+t}+\dots +x^{r+s+t-1}\\
&=& \frac{x^r-1}{x-1}+x^{r+t}\frac{x^s-1}{x-1}.
\end{eqnarray*}
Or equivalently, \begin{equation}\label{eqn:pax} (x-1) \ga_A(x) = (x^r
- 1) + x^{r+t} (x^s - 1).
\end{equation} Let $\gcd(n, r, s) = k>1$. Then it can be easily checked that
$\ze_n^{n/k}$ is a root of Equation~(\ref{eqn:pax}). Thus, $X$ is
singular. This completes the proof of the first part.

{\bf Proof of Part~\ref{lem:rst:2}.\ref{lem:rst:2.1}:} Let us assume that $\gcd(n,s) = 1$.
Also, let us assume that there exists a positive integer $d \ge 2$ such that $d
\mid t$ and $s=\ell r$, for some positive integer $\ell \equiv -1 \pmod d$. So,
there exists $\be \in \Z$ such that $ \ell = \be d - 1$.
In this case, using Equation~(\ref{eqn:pax}), we get
\begin{eqnarray*} (\ze_n^{(n/d)} - 1) \ga_A(\ze_n^{(n/d)}) &=& (\ze_n^{(r
n/d)} - 1) \left( 1 + \ze_n^{(r+t) n/d}
\frac{\ze_n^{\ell (r n/d)} - 1}{\ze_n^{(r n/d)} - 1} \right) \\
&=& (\ze_n^{r n/d} - 1) \left( 1 + \ze_n^{(r+t) n/d}
\frac{\ze_n^{-(r n/d)} - 1}{\ze_n^{(r n/d)} - 1} \right) \\
& = & (\ze_n^{r n/d} - 1) \left( 1 - \ze_n^{(t n/d)} \right).
\end{eqnarray*}
As $d \mid t$,  $ \; \ga_A(\ze_n^{n/d}) = 0$. That is, we get the required
result in this case as well.

{\bf Proof of  Part~\ref{lem:rst:2}.\ref{lem:rst:2.2}:} Let us assume that
$\gcd(n,s) = 1$, $n = 2m$. Also, let us assume that there exists an even
positive integer $d $ such that $r+t$ is an odd multiple of $\frac{d}{2}$  and
$s=\ell r$, for some positive integer $\ell \equiv 1 \pmod d$. Then there exists
$\be \in \Z$ such that $ \ell = \be d + 1$. In this case, using
Equation~(\ref{eqn:pax}), we get
\begin{eqnarray*} (\ze_n^{(n/d)} - 1) \ga_A(\ze_n^{(n/d)}) &=& (\ze_n^{(r
n/d)} - 1) \left( 1 + \ze_n^{(r+t) n/d}
\frac{\ze_n^{\ell (r n/d)} - 1}{\ze_n^{(r n/d)} - 1} \right) \\
&=& (\ze_n^{r n/d} - 1) \left( 1 + \ze_n^{(r+t) n/d}
\frac{\ze_n^{(r n/d)} - 1}{\ze_n^{(r n/d)} - 1} \right) \\
& = & (\ze_n^{r n/d} - 1) \left( 1 + \ze_n^{(r+t) n/d)} \right).
\end{eqnarray*}
Thus, under the given conditions, the corresponding digraph $X$ is
singular.

Hence, the proof of the lemma is complete.
\end{proof}

Thus, the above result gives conditions under which the generalized $(r,s,t)$-
digraphs, for non-negative values of $r, s$ and $t$, are singular. We will now
define another class of circulant digraphs and obtain
conditions under which the circulant digraphs are singular. These graphs are
also a generalization of the graphs studied in  Lemma~\ref{lem:k}.

Let $i, j, k$ and $\ell$ be non-negative integers such that $j > \ell$ and $k j
+ i + \ell < n$. Consider a class of circulant digraphs, denoted
$C_n^{i,j,k,\ell}$, that has
$\ga_{A(C_n^{i,j,k,\ell})}(x)= \sum\limits_{t=0}^{k}\sum\limits_{s=i}^{i+\ell}
x^{s + tj}$  as its representer polynomial.
Then
\begin{eqnarray}
\ga_{A(C_n^{i,j,k,\ell})}(x)&=& x^i(1+x+\cdots+x^\ell)(1+x^j+x^{2j}+\cdots +
x^{kj})\nonumber\\
 &=& x^i
\frac{x^{\ell+1}- 1}{x-1} \; \cdot \; \frac{x^{(k+1)j } - 1}{x^j - 1}
\nonumber
\\ &=&x^i \prod_{s \mid \ell+1, s\ne 1} \Phi_s(x) \; \cdot \; \prod_{t \mid
(k+1)j, t\nmid j} \Phi_t(x).
\end{eqnarray}
Hence, we have the following theorem which we state without proof.

\begin{theorem}\label{thm:gen}
Let $i, j, k$ and $\ell$ be non-negative integers with $j > \ell$ and $k j
+ i + \ell < n$. Then the
circulant digraph $C_n^{i,j,k,\ell}$, defined above,  is
singular if and only if either $\gcd(\ell + 1,n) \ge 2$ or
 $\gcd(k+1, \frac{n}{\gcd(n,j)}) \ge 2$.
\end{theorem}

\begin{rem}
Note that we can vary the non-negative integers $i, j, k$ and $\ell$
to define quite a few class of circulant digraphs. For example, it
can be seen that the graphs $G(r,t)$ that are given by
Doob~\cite{doob} are a particular case of the above class. Also, it
can be easily verified that Theorem~\ref{thm:gen} is a
generalization of Lemma~\ref{lem:k}.
\end{rem}

\section*{Conclusion}
In the first section, we have obtained necessary and sufficient conditions for a few known classes of circulant graphs/digraphs to be singular.   We found these
necessary and sufficient conditions by using Lemma~\ref{lem:Ga}.
The graphs/digraphs that were  studied in Section~$2$ have been generalized to $(r,s,t)$-circulant digraphs for non-negative integers $r,s$ and $t$, and the circulant digraph
$C_n^{i,j,k,l}$, under certain restrictions. A necessary and sufficient condition
for the digraphs $C_n^{i,j,k,l}$ to be singular is also obtained. Some necessary
conditions are given under which the $(r,s,t)$-circulant digraphs  are singular.

 It will be nice to obtain necessary and sufficient conditions for the generalized $(r,s,t)$-digraphs to be singular.


\begin{thebibliography}{20}

\bibitem{apostol} Tom M. Apostol. ``Introduction to Analytic Number theory,"
{\it Springer-Verlag}, New York, 1976.



\bibitem{biggs} N. Biggs. ``Algebraic Graph Theory," Cambridge Tracts in
Mathematics, 67, Cambridge Univ. Press, New York and London, 1974.

\bibitem{bondy} J. A. Bondy and U. S. R. Murty. {\em Graph Theory  with Applications}.
The Macmillan Press Ltd., London, 1976.

\bibitem{davis}
Philip J. Davis. ``Circulant matrices," A Wiley-interscience
publications, 1979.

\bibitem {doob} Michael Doob. `Circulant graphs with $\det(-A(G)) = - \deg(G)$:
codeterminants with $K_n$,' Linear algebra and its applications, 340: 87 - 96, 2002.


\bibitem{D:F}
David S. Dummit and Richard M. Foote. ``Abstract Algebra," Second
Edition, John Wiley and Sons, 1999.


\bibitem{gel} D. Geller, I. Kra, S. Popescu and S. Simanca. `On circulant matrices,' (http://www.math.sunysb.edu/ $\!\tilde\!$ sorin/eprints/circulant.pdf).

\bibitem{odlyzko} R. P. Kurshan and A. M. Odlyzko. `Values of
 cyclotomic polynomials at roots of unity,' Math. Scand. 49,
 15 - 35, 1981.

\bibitem{V:P} Victor V. Prasolov. ``Polynomials," Springer, 2001.

\bibitem{leo} Leonor Aquino-Ruivivar. `Singular and Nonsingular
Circulant Graphs,' Journal of Research in Science, Computing and
Engineering (JRSCE), Vol. 3, No. 3,  2006.



\bibitem{S:R}S. R. Searle. `On Inverting Circulant Matrices,'
Linear algebra and its applications, 25: 77 - 89, 1979.



\end{thebibliography}
\end{document}